\documentclass{article}

\usepackage{graphicx} 
\usepackage{float}
\usepackage[main=english]{babel}
\usepackage{amsmath, amsthm}
\usepackage{amsfonts} 
\usepackage{enumitem}
\usepackage{appendix} 
\usepackage{makecell}
\usepackage{longtable}
\newtheorem{theorem}{Theorem}
\newtheorem{lemma}[theorem]{Lemma} 
 
\newtheorem{example}[theorem]{Example} 
\newtheorem{definition}[theorem]{Definition} 
\newtheorem{remark}[theorem]{Remark}

\title{A Resource Allocation Tournament\thanks{This work was supported by the Israel Science Foundation [Grant 2718/25].}}
\author{Adam Nathaniel\footnote{Department of Statistics and Operations Research, School of Mathematical Sciences, Tel Aviv University. e-mail: adamnathaniel6@gmail.com.}~ and Refael Hassin\footnote{Department of Statistics and Operations Research, School of Mathematical Sciences, Tel Aviv University. e-mail: hassin@tauex.tau.ac.il.}}

\date{2026}

\begin{document}

\setcellgapes{3pt}   
\makegapedcells       

\maketitle

\begin{abstract}

We introduce and analyze a round-robin simultaneous resource allocation contest, motivated by sports leagues and other competitive environments where agents must distribute limited resources across multiple encounters. Players allocate resources in advance to matches against all opponents, and match outcomes are determined by relative allocations. We study two payoff schemes: total points, where players accumulate rewards for each win, and winner-takes-all, where only the tournament winner is rewarded. For three players, we provide complete equilibrium characterizations under both schemes. For four or more players, we show that the total points payoff admits Blotto-type equilibria, while the winner-takes-all payoff requires numerical and algorithmic analysis. To address this, we develop a best-response dynamics algorithm with stochastic inertia and use it to approximate equilibria in games with four and five players. 

\end{abstract}


\section{Introduction}

We define and analyze a game motivated by sports leagues. We model a round-robin tournament as a contest in which teams compete for a prize by distributing costly resources, such as energy, across matches. The analysis examines how teams should optimally allocate their resources throughout the course of the competition.

The scope of resource allocation competitions goes far beyond sports. Similar strategic problems arise in economics, where firms must decide how to allocate limited budgets across research and development projects or advertising campaigns. In political science, election campaigns can be viewed as contests in which parties distribute their resources over time to maximize electoral success. Analogous dynamics also appear in international relations, for instance, in arms races, where countries repeatedly allocate effort to maintain or gain advantage. Thus, while sports provide natural and intuitive motivation, the underlying allocation problem is relevant in a wide range of settings.

In multi-stage contests, the resource allocation can be done either simultaneously, where players commit to their allocations for all stages in advance without knowing the outcomes of earlier stages, or sequentially, where allocations are chosen stage by stage, with players observing the outcomes of previous stages before making subsequent decisions. Although sequential contests are more common, simultaneous resource allocation problems have many real-life applications. In what follows, we discuss several such applications, before presenting our simultaneous resource allocation framework.

In sports, a clear example arises in pre-season planning, where teams design training-load or effort-allocation programs for the entire season in advance, committing ex ante to a distribution of resources across competitions or time periods, with only limited ability to adjust once the season begins. Similarly, roster construction and budget allocation across positions before the season begins can be viewed as a simultaneous resource allocation problem, since most decisions are made prior to observing match outcomes.

Simultaneous auctions have also been intensively studied. The literature includes theoretical analysis of equilibrium behavior and social welfare properties in simultaneous auctions (Krishna and Rosenthal, 1996; Feldman et al., 2013), as well as comparisons between simultaneous and sequential auctions from the perspective of the auction designer (Hausch, 1986). In parallel, a substantial body of work is devoted to the design and implementation of simultaneous auctions for spectrum allocation (McAfee and McMillan, 1996; Milgrom, 2000; Cramton, 2013). Algorithmic approaches to bidding in simultaneous auctions have also been studied (Greenwald and Boyan, 2001). 

In contrast to our framework, in these auction settings, only winning bidders incur payments, whereas unsuccessful bidders do not bear any cost. A notable exception is the waiting-line auction (Holt and Sherman, 1982), which studies the allocation of commodities on a first-come, first-served basis. Consumers compete along a non-price dimension, time, closely resembling the non-price dimension of effort in our framework.

A sports-related simultaneous auction can be seen in the 2026 FIFA World Cup ticketing process, where fans enter a simultaneous ballot, submitting binding purchase requests for all matches in which they are interested. While such load management allocation problems can also be addressed through sequential mechanisms in cyclic queueing games (Hassin and Nathaniel, 2021), the choice of a simultaneous ballot introduces unique strategic complexities. Specifically, fans may fear winning too many tickets, i.e., obtaining more tickets than they can afford or wish to purchase, which may lead them to bid conservatively. A closely related concern appears in McAfee and McMillan (1996), which shows that in simultaneous auctions, budget-constrained bidders strategically shade their bids due to the risk of winning more items than desired. Although our model also incorporates a budget constraint, the fact that all players incur costs regardless of winning eliminates incentives to limit the number of wins.

Additional simultaneous mechanisms in which players incur costs regardless of the outcome include rent seeking activities (Tullock, 1980; Congleton et al., 2008; Fonseca, 2009). In the efficient rent seeking model, simultaneity is imposed under the presumption that all players have an incentive to act as first movers. Early work (Krueger, 1974) emphasizes the competitive nature of rent seeking, demonstrating that the associated welfare losses may substantially exceed those generated by traditional trade restrictions. Subsequent research (Baye and Hoppe, 2003; Ahn et al., 2010) highlights the strategic equivalence between rent seeking contests and tournament-style games, such as innovation tournaments in which only the best entrepreneur can go to the market. Although the literature predominantly adopts a simultaneous-move framework, Glazer and Hassin (2000) study sequential rent seeking and show that early movers do not necessarily earn higher profits and that aggregate profits may be lower than under simultaneity.

Another area close to simultaneous resource allocation contests is the course bidding mechanism, where students use virtual points to bid for limited seats in popular courses. Krishna et al. (2004) report a field experiment at the University of Michigan Business School, comparing different course allocation systems. Krishna and Unver (2008) complement this with laboratory evidence in controlled settings. Sonmez and Unver (2010) highlight the conceptual problem that bids used both for preference revelation and priority allocation can distort outcomes. More recently, Juthamanee et al. (2021) and Laokondee et al. (2021) applied deep learning models to optimize course allocation. 

Together, these applications illustrate the practical relevance of strategic resource allocation in simultaneous strategic interactions. Rent seeking activities resemble our framework in settings where agents optimize tournament rank, while course bidding mechanism resembles the setting where agents optimize individual outcomes. 

At a formal level, such problems are naturally studied within the broader framework of contest theory, to which we now turn.

Contest theory (Glazer and Hassin, 1988; Moldovanu and Sela, 2001) is a branch of game theory that studies situations where players expend costly resources (effort, money, time, etc.) to compete for a prize. As in auctions, the outcomes are determined by relative bids or efforts. The key difference from auctions is that, in contests, all participants bear the cost of their actions, even if they do not win the prize.

Contests may differ in various aspects, such as the tournament format and the mechanism used to determine match outcomes. In particular, in standard profit-maximization contests, players aim to maximize their profit, typically defined as the expected reward minus their chosen cost. In contrast, in resource allocation contests, players are constrained to spend their entire resources (or cost), and thus aim only to maximize their expected reward through the optimal allocation of that fixed budget. This distinction is especially relevant in sports economics, where professional soccer clubs have been shown to follow win-maximization strategies subject to budget constraints rather than pure profit-maximization (Garcia-del-Barrio and Szymanski, 2009).

To the best of our knowledge, our model, formally presented in the next section, has not previously been studied. In what follows, we review related work.

Resource allocation design has been studied in two-player contests with the Tullock mechanism (Tullock, 1980), where if Player $i$ allocates $x_i$ resources in a match, their probability of winning the match is $\frac{x_i}{x_1+x_2}$. Sela and Erez (2013) analyze contests in which the two players compete against each other in $n$ different stages, having winning values for each stage. Sela (2023) analyzes the best-of-two and best-of-three tournaments, where only the winner of the tournament receives a prize.

The Colonel Blotto game (Borel, 1921) can be seen as a variation of the two-player resource allocation contest with the all-pay mechanism, where if Player $i$ allocates $x_i$ resources in a match, they win the match against Player $j$ if $x_i>x_j$. The two players compete against each other in $n$ simultaneous matches on $n$ battlefields, and their reward is equal to the total number of matches they win. In its multiplayer version (Boix-Adsera et al., 2020), $k$ players allocate resources over $n$ battlefields, where each player decides the amount of resources to allocate to each battlefield.

Round-robin tournaments with more than two players, where players do not have resource budgets as in our model, have also been studied. Krumer et al. (2017) analyze round-robin tournaments with three and four symmetric players, where the matches are all-pay contests, and only the winner of the tournament receives a prize. The matches are played sequentially, and it is shown that asymmetry in the order of the matches leads to asymmetry in the winning probability of the tournament among the players. Dagaev and Zubanov (2022) consider a round-robin sequential resource allocation tournament with three players and compare different distributions of prizes, some of them leading to equal payoffs in equilibrium. 

In addition to theoretical work, Avrahami et al. (2013) conducted an experimental study of resource allocation in a two-player contest over many fields, comparing behavior in symmetric and asymmetric setups. They show that weaker players often abandon some fields to focus resources and maintain a chance of occasionally winning.

In contrast to the models surveyed above, we consider a multiplayer round-robin resource allocation tournament with simultaneous all-pay matches.

The remainder of this paper is organized as follows. Section 2 presents our model and suggests two payoff schemes: total points and winner-takes-all. Sections 3-6 analyze each payoff scheme, first in the three-player game and then in the game with four or more players. Section 7 concludes the paper.

\section{Model}

We consider a symmetric game with $n$ players. The players compete in a round-robin tournament, facing each of the other $n-1$ players exactly once. Hence, each player participates in $n-1$ matches.

Each player is endowed with a total of one unit of resources to be allocated across their matches. Formally, a strategy is an allocation vector $(x_1, \ldots, x_{n-1})$ with $\sum_{i=1}^{n-1} x_i = 1$, where $x_i$ denotes the resources assigned to the $i$-th match. 

Players choose their allocation strategy simultaneously and in advance, without knowing the identity of their opponents in each match or the order in which the matches occur. In particular, since all strategies are determined before the competition begins, players cannot observe the outcomes of any matches.

The outcome of a match is determined by the relative resource allocations: the player who allocates more resources wins the match and receives $1$ point; if both allocate the same amount, the match ends in a draw and each player receives $0.5$ point.

The payoffs of the players in the tournament are based on their points. We consider two payoff schemes: 

\begin{itemize} 
    \item \textbf{Total points payoff}: The payoff for each player equals the number of points accumulated throughout the tournament.
    \item \textbf{Winner-takes-all payoff}: A reward of one unit is shared equally among all players with the highest total of points at the end of the tournament. 
\end{itemize}

Since the tournament schedule is random, we evaluate the expected payoffs in all possible permutations of opponents and match orders. Thus, in the total points payoff scheme we calculate each player’s expected total points, and in the winner-takes-all payoff scheme we calculate the winning probability of each player.

We look for a symmetric Nash equilibrium strategy. Because the two-player case is trivial, we begin with the three-player game for each payoff scheme.

\section{Total points payoff: three players}
In this section, we consider a three-player game in which each player faces the other two exactly once and allocates their resources between the two matches. The players aim to maximize their expected total number of points earned throughout the tournament.

Pure strategies are of the form $(x, 1 - x)$. Since players must decide on their allocation strategy without knowing the identity of their opponent in a given match, we may, without loss of generality, restrict attention to $x \in \left[\frac{1}{2}, 1\right]$ to avoid duplication of strategies.

\begin{theorem}\label{thm:basic}
Any resource allocation strategy of the form $(x, 1 - x)$, where $x \in \left[\frac{1}{2}, 1\right]$, is a symmetric equilibrium.
\end{theorem}

\begin{proof}
Denote my strategy by $(y, 1 - y)$, where $y \in \left[\frac{1}{2}, 1\right]$, and denote the strategy of each of the other two players by $(x, 1 - x)$, where $x \in \left[\frac{1}{2}, 1\right]$. Let $P_0$ be my expected payoff and $P_1$ be the expected payoff of each of the other two players.

A total of three matches take place and, therefore, the sum of the players' payoffs satisfies \( P_0 + 2P_1 = 3 \). In a symmetric equilibrium, we have \( P_0 = P_1 = 1 \).

We show that any response strategy $(y, 1 - y)$ yields the same payoff $P_0=1$ and therefore it is also a best response.

In the case where \( x = y \), it is straightforward that \( P_0 = P_1 = 1 \).

A strategy with \( y < x \) satisfies \( x > y \geq 1 - y > 1 - x \), and I win the match to which I allocate \( y \) units with probability \( \frac{1}{2} \), since I win against \( 1 - x \), lose against \( x \) and encounter each of them with equal probability. 

Independently, as I play against a different player, I also win the match to which I allocate \( 1 - y \) units with probability \( \frac{1}{2} \), by the same reasoning: I win against \( 1 - x \), lose against \( x \), and face each of them with equal probability.

Similarly, strategy \( y > x \) satisfies \( y > x \geq 1 - x > 1 - y \), and I win the match to which I allocate \( y \) units with probability \( 1 \), and I win the match to which I allocate \( 1 - y \) units with probability \( 0 \).

In conclusion, in the case \( y < x \), my expected payoff is \( P_0 = \frac{1}{2} + \frac{1}{2} = 1 \), and in the case \( y > x \), my expected payoff is \( P_0 = 1+0=1 \) as well.

Therefore, my strategy $(y, 1 - y)$ is always a best response against $(x, 1 - x)$. Specifically, the strategy with $y = x$ is a best response, and therefore $(x, 1 - x)$ is a symmetric equilibrium.
\end{proof}

\begin{remark}
More generally, not only does any pure strategy constitute a symmetric equilibrium, but any profile of pure strategies constitutes an equilibrium, and therefore any mixed strategy constitutes a symmetric equilibrium.
\end{remark}

\begin{proof}
We have already shown that, in the case that two players play the same pure strategy and one player plays a different pure strategy, the payoffs of all players are equal. In the case of three players playing three different pure strategies, there are eight equally likely configurations corresponding to each player randomly allocating their resources between the two opponents. By a calculation analogous to that in the proof of Theorem~\ref{thm:basic}, it can be shown that the payoffs of all players are equal in this case as well.
\end{proof}

\section{Total points payoff: four or more players}
In this section, we consider a game with four or more players. The players aim to maximize their expected total number of points earned throughout the tournament. 

We begin the section with an analysis of the four-player game in which each player faces the other three exactly once and allocates their resources over the three matches. 

Pure strategies are of the form $(x_1, x_2, x_3)$, where $x_1 + x_2 + x_3 = 1$. Since players must decide on their allocation strategy without knowing the identity of their opponent in a given match, we may, without loss of generality, restrict attention to strategies satisfying $x_1 \geq x_2 \geq x_3$ to avoid duplication.

Unlike the three-player variation, in the case of four players, not every strategy of the form \( (x_1, x_2, x_3) \) constitutes a symmetric equilibrium. In fact, none of them does, and the game admits no pure symmetric equilibrium.

\begin{lemma}\label{thm:lemma1}
In the four-player game, there does not exist a pure symmetric equilibrium.
\end{lemma}

\begin{proof}

Let \( (x_1, x_2, x_3) \) be the strategy played by each of the other three players, where \( x_1 \geq x_2 \geq x_3 \) and \( x_1 + x_2 + x_3 = 1 \). Then, for sufficiently small \( \epsilon > 0 \), the strategy \( (x_1 - 2\epsilon, x_2 + \epsilon, x_3 + \epsilon) \) yields a strictly higher expected payoff for me than the strategy \( (x_1, x_2, x_3) \).

Consider the case \( x_1 > x_2 > x_3 \). Let $\epsilon = \frac{1}{4} \min(x_1 - x_2,\, x_2 - x_3)$. Then the following inequalities hold: $x_1 > x_1 - 2\epsilon > x_2 + \epsilon > x_2 > x_3 + \epsilon > x_3$.

Denote my expected payoff by $P_0$, and denote the expected payoff of each of the other three players by $P_1$. A total of six matches are played, so the sum of all payoffs satisfies \( P_0 + 3P_1 = 6 \). Clearly, under my strategy \( (x_1, x_2, x_3) \) we have \( P_0 = P_1 = 1\frac{1}{2} \).

Now consider my strategy \( (x_1 - 2\epsilon,\, x_2 + \epsilon,\, x_3 + \epsilon) \). I win the match to which I allocate \( x_1 - 2\epsilon\), \( x_2 + \epsilon\) and \( x_3 + \epsilon\) units with probabilities \( \frac{2}{3} \), \( \frac{2}{3} \) and \( \frac{1}{3} \), respectively/ 

Therefore, my expected payoff under this strategy is $P_0 = \frac{2}{3} + \frac{2}{3} + \frac{1}{3} = 1\frac{2}{3} > 1\frac{1}{2} $, as claimed.

Similar arguments apply in the cases where \( x_1 = x_2 \), \( x_2 = x_3 \), or \( x_1 = x_2 = x_3 = \frac{1}{3} \). Against the strategy \((x, x, 1 - 2x)\), where \( x \in \left(\frac{1}{3}, \frac{1}{2}\right] \), I can respond with \((x + \epsilon, x - 2\epsilon, 1 - 2x + \epsilon)\), and receive a payoff of $1\frac{2}{3}$. Against the strategy \((1 - 2x, x, x)\), where \( x \in \left[0, \frac{1}{3}\right) \), I can respond with \((1 - 2x - 2\epsilon, x + \epsilon, x + \epsilon)\), and receive a payoff of $2$. Finally, against \((\frac{1}{3}, \frac{1}{3}, \frac{1}{3})\), I can respond, for example, with \((\frac{1}{2}, \frac{1}{2}, 0)\), and receive a payoff of $2$.

\end{proof}

The essence of the proof is that it is profitable to be better than the opponents in two out of the three coordinates and worse in the third coordinate. The following example shows how the better response dynamics described above can lead to cyclical behavior.

\begin{example}
Suppose that the other three players play $(0.34, 0.33, 0.33)$.  This is not an equilibrium, as for me a better response is, for example, $(0.43, 0.43, 0.14)$. If they adopt this new strategy, my response becomes $(0.53, 0.24, 0.23)$. They may continue to follow, allowing me, in turn, to return to $(0.34, 0.33, 0.33)$ and gain an advantage once again.
\end{example}

Therefore, we seek a mixed equilibrium strategy. Specifically, we define a property, construct a distribution that satisfies this property, and prove that the property constitutes a sufficient condition for equilibrium.

\begin{definition}
For a game with \( n+1 \)-players, a mixed strategy $(X_1, X_2, \ldots, X_n)$ is called \emph{Blotto} if its joint probability distribution satisfies the following conditions:
\begin{enumerate}
    \item \( \displaystyle \sum_{i=1}^{n} X_i = 1 \),
    \item Each marginal distribution is uniform on the interval \( \left[0, \frac{2}{n}\right] \); that is, \( X_i \sim \mathrm{U}\left[0, \frac{2}{n}\right] \) for all \( i \in \{1, \ldots, n\} \).
\end{enumerate}
\end{definition}

A Blotto strategy was first presented by Borel and Ville (1938). In our setup, this corresponds to an \( n+1 \)-player game, whereas the paper analyzes the Colonel Blotto game, involves two players competing over $n$ battlefields. 

Roberson (2006) and Weinstein (2012) provide an elegant construction of the case $n=3$. Consider a trivariate distribution that places a uniform probability mass \( \frac{1}{2} \) on each of the following two line segments: from \( (0, \frac{2}{3}, \frac{1}{3}) \) to \( (\frac{2}{3}, \frac{1}{3}, 0) \), and from \( (0, \frac{1}{3}, \frac{2}{3}) \) to \( (\frac{2}{3}, 0, \frac{1}{3}) \). This defines a Blotto strategy for the case $n=3$. Weinstein (2012) further generalizes the construction to arbitrary $n$.

\begin{theorem}
For an \( n+1 \)-player game, any Blotto strategy constitutes a mixed symmetric equilibrium.
\end{theorem}

\begin{proof}
Assume that the other $n$ players use a common mixed Blotto strategy and that I play the pure strategy \( (x_1, x_2, \ldots, x_n) \), where \( x_1 \geq x_2 \geq \dots \geq x_n \) and \( x_1 + x_2 + \ldots + x_n = 1 \). We will show that this strategy is a best response if and only if \( x_1 \leq \frac{2}{n} \).  

By the definition of a Blotto strategy, in each of my $n$ matches, my opponent randomizes their allocation according to the uniform distribution over the interval \( [0, \frac{2}{n}] \). Therefore, my expected payoff is 

\[
P_0 = \sum_{i=1}^{n} \Pr(Y_i < x_i \mid Y_i \sim \text{U}[0, \tfrac{2}{n}]).
\]

For any \( 0 \leq x \leq 1 \) and \( Y \sim \text{U}[0, \tfrac{2}{n}] \), we have:

\[
\Pr(Y < x) = 
\begin{cases} 
\frac{n}{2}x & \text{if } 0 \leq x \leq \tfrac{2}{n}, \\
1 & \text{if } \tfrac{2}{n} < x \leq 1.
\end{cases}
\]

Therefore, if \( x_1 \leq \frac{2}{n} \), my expected payoff is

\[
P_0 = \frac{n}{2} x_1 + \frac{n}{2} x_2 + \ldots + \frac{n}{2} x_n = \frac{n}{2},
\]

and if \( x_k > \frac{2}{n} \geq x_{k+1}\), then my expected payoff is

\[
P_0 = k + \frac{n}{2} (x_{k+1} + \ldots +  x_n) < \frac{n}{2} (x_1 + \ldots + x_n) = \frac{n}{2}.
\]

Therefore, we conclude that a pure strategy \( (x_1, x_2, \ldots, x_n) \) is a best response if and only if \( x_1 \leq \frac{2}{n} \). As a result, any mixed strategy that randomizes between pure strategies, all of which have \( x_1 \leq \frac{2}{n} \), is also a best response.

The marginal distributions of the variables in the Blotto strategy are all $\mathrm{U}\left[0, \frac{2}{n}\right]$, ensuring that no individual value exceeds \( \frac{2}{n} \). Therefore, any Blotto strategy constitutes a best response to  itself, and thus forms an equilibrium.
\end{proof}

\begin{remark}
More generally, any strategy profile in which each player selects a Blotto strategy (not necessarily identical for all players) constitutes a mixed equilibrium.
\end{remark}

\section{Winner-takes-all payoff: three players}
In this section, we consider a three-player game in which each player faces each of the other two exactly once and allocates their resources between the two matches. The players aim to maximize their chances of winning the tournament.

Pure strategies are of the form $(x, 1 - x)$. Since players must decide on their allocation strategy without knowing the identity of their opponent in a given match, without loss of generality, we restrict attention to $x \in \left[\frac{1}{2}, 1\right]$ to avoid duplication of strategies.

Unlike the variation that maximizes expected points, in the case of maximizing the probability of winning, not every pure strategy of the form \( (x, 1 - x) \) constitutes a symmetric equilibrium. In this section, we fully characterize the payoff function, representing the player's chance of finishing first, and prove that the strategy \( (\frac{1}{2}, \frac{1}{2}) \) is the unique symmetric equilibrium. 

\begin{lemma}\label{thm:main0}
Assume that the three players play pure strategies $(x, 1-x)$, $(y, 1-y)$, and $(z, 1-z)$. Without loss of generality, suppose that $x \geq y \geq z \geq \frac{1}{2}$. The payoffs of the players, which we denote as $P_x$, $P_y$, and $P_z$, satisfy $P_x+P_y+P_z=1$ and are given in Table~\ref{tab:probabilities}.

\begin{table}[ht]
\centering
\[
\begin{array}{|c|c|c|c|}
\hline
\textbf{case} & \textbf{$P_x$} & \textbf{$P_y$} & \textbf{$P_z$} \\ \hline
x > y > z & 1/6 & 5/12 & 5/12 \\ \hline 
x = y > z & 7/24 & 7/24 & 5/12 \\ \hline 
x > y = z > 1/2 & 1/12 & 11/24 & 11/24 \\ \hline 
x > y = z = 1/2 & 0 & 1/2 & 1/2 \\ \hline 
x = y = z & 1/3 & 1/3 & 1/3 \\ \hline 
\end{array}
\]
\caption{Payoffs in the three-player game.}
\label{tab:probabilities}
\end{table}

\end{lemma}

\begin{proof}

In each of the five cases, there are eight equally likely configurations corresponding to each player randomly allocating their resources between the two opponents. Specifically, Player 1 allocates $x$ units of resources against Player 2 or against Player 3; Player 2 allocates $y$ units of resources against Player 1 or against Player 3; and Player 3 allocates $z$ units of resources against Player 1 or against Player 2.

In each case-configuration combination, there are three matches, with deterministic resource allocations and outcomes, leading to a deterministic payoff for each player. The expected payoff of a player in a given case is defined as the average of their payoffs across all case-configuration combinations corresponding to that case.

For example, consider the case $x > y > z$ and the configuration in which Player 1 allocates $x$ against Player 2, Player 2 allocates $y$ against Player 1, and Player 3 allocates $z$ against Player 1.

We have $x > y > z \geq \frac{1}{2} \geq 1-z > 1-y > 1-x$. Therefore, Player 1 wins against Player 2, Player 3 wins against Player 1, and Player 3 wins against Player 2. Consequently, in this case-configuration, Player 3 wins the tournament, and the payoffs are $P_x=P_y=0$ and $P_z = 1$.

The other case-configuration combinations are analyzed similarly; the complete analysis is provided in the appendix.
\end{proof}

Lemma~\ref{thm:main0} provides a complete characterization of the payoff function for pure strategies. We can use it to find the unique symmetric equilibrium.

\begin{theorem}\label{thm:main}
The strategy \( \left(\tfrac{1}{2}, \tfrac{1}{2}\right) \) is the unique symmetric equilibrium.
\end{theorem}

\begin{proof}

First, we prove that the strategy \( \left(\tfrac{1}{2}, \tfrac{1}{2}\right) \) is an equilibrium. Assume that the other two players play \( \left(\tfrac{1}{2}, \tfrac{1}{2}\right) \), and that I play $(x, 1-x)$, where $x \geq \frac{1}{2}$. According to Lemma~\ref{thm:main0}, if $x=\frac{1}{2}$ my payoff is $\frac{1}{3}$  and if $x>\frac{1}{2}$ my payoff is $0$. Therefore, \( \left(\tfrac{1}{2}, \tfrac{1}{2}\right) \) is an equilibrium. We will prove that it is unique.

Consider another strategy $(X, 1-X)$ of the other two players, where $X$ is a random variable over pure strategies and may be constant. We will prove that $(X, 1-X)$ is not an equilibrium.

If I play $(X, 1-X)$, my payoff is $\frac{1}{3}$. If I play \( \left(\tfrac{1}{2}, \tfrac{1}{2}\right) \), my payoff is an expectation of my payoffs over all realizations $(x_1, 1-x_1)$ and $(x_2, 1-x_2)$ of the other two players, with $x_1,x_2 \geq \tfrac{1}{2}$.

From Lemma~\ref{thm:main0},  for any values of $x_1$ and $x_2$, my profit from playing \( \left(\tfrac{1}{2}, \tfrac{1}{2}\right) \), denoted $P_z$ in the lemma, is greater than or equal to $\frac{1}{3}$ and is strictly greater than $\frac{1}{3}$ unless $x_1 = x_2 = \frac{1}{2}$. 

Since their strategy is not \( \left(\tfrac{1}{2}, \tfrac{1}{2}\right) \), with positive probability we have $x_1 \neq \frac{1}{2}$. Therefore, with expectation over all realizations, my profit from the strategy \( \left(\tfrac{1}{2}, \tfrac{1}{2}\right) \) is strictly greater than my profit from the strategy $(X, 1-X)$. Hence, $(X, 1-X)$ is not a best response against $(X, 1-X)$, and thus not an equilibrium.

\end{proof}

Although mixed strategies such as $(X, 1-X)$, where $X \sim U[0,1]$, may appear as natural candidates for a symmetric equilibrium, we have shown that none of them constitutes one. In addition, we provide a characterization of the best responses to mixed strategies of this form.

\begin{remark}\label{thm:main2}
Assume that the other two players allocate their resources independently as $(X, 1-X)$, where $X \sim f_X$, a continuous density function over $[0,1]$. Assume that I choose my allocation deterministically as $(y, 1-y)$, where $y \in [\frac{1}{2},1]$. My payoff, denoted as $P_0$, is given by 
\[
P_0 = \frac{1}{6} + a_y (1 - a_y) \quad \text{where} \quad a_y = \frac{1-\Pr(1-y < X < y)}{2}.
\]
My strategy $(y, 1-y)$ is a best response if and only if $\Pr(1-y < X < y) = 0$. My strategy $(\frac{1}{2}, \frac{1}{2})$ is always a best response.
\end{remark}

The proof for Remark~\ref{thm:main2} is provided in the appendix. 

\begin{example}
Assume that the other two players allocate their resources independently as $(X, 1-X)$, where $X \sim U[0,1]$. Assume that I choose my allocation deterministically as $(y, 1-y)$, where $y \in [\frac{1}{2},1]$. My payoff, denoted as $P_0$, is given by 
\[
P_0 = \frac{1}{6} + y(1-y).
\]
Since $a_y=\frac{1}{2}(1-(y-(1-y)))=1-y$.
My unique best response is $(\frac{1}{2}, \frac{1}{2})$.
\end{example}

\section{Winner-takes-all payoff: four or more players}
In this section, we consider a game with four or more players. The players aim to maximize their chances of winning the tournament.

We begin the section with an analysis of the four-player game in which each player faces the other three exactly once and allocates their resources across the three matches. 

Pure strategies are of the form $(x_1, x_2, x_3)$, where $x_1 + x_2 + x_3 = 1$. Since players must decide on their allocation strategy without knowing the identity of their opponent in a given match, without loss of generality, we restrict attention to strategies satisfying $x_1 \geq x_2 \geq x_3$ to avoid duplication.

An analytical solution for this variation may be challenging to derive, as the payoff of one player also depends on the interactions among the other players, and the number of such interactions increases with the number of players. Moreover, unlike the case with three players, the variation with four players has no pure equilibrium.

We solve this variation numerically. We present a geometric representation of this game in Section~\ref{sec:geometry} and its numerical discretization in Section~\ref{sec:discretization}. Then, we prove that there is no pure equilibrium in Section~\ref{sec:pure} and discuss the analysis of mixed strategies in Section~\ref{sec:mixed}. Finally, we use methods from algorithmic game theory to find an $\varepsilon$-equilibrium in Section~\ref{sec:algorithmic} and extend our solution to the five-player game in Section~\ref{sec:five}.

\subsection{Geometric representation}
\label{sec:geometry}
The possible strategies correspond to points satisfying the constraints $x_1 \geq 0$, $x_2 \geq 0$, $x_3 \geq 0$, and $x_1 + x_2 + x_3 = 1$, which define a two-dimensional triangle bounded by the lines $x_1 = 0$, $x_2 = 0$, and $1 - x_1 - x_2 = 0$. However, this triangle contains strategies that are strategically equivalent. For example, the vertices of this triangle are $(1,0)$, $(0,1)$, and $(0,0)$, corresponding to the equivalent strategies $(1,0,0)$, $(0,1,0)$, and $(0,0,1)$, respectively.  

To avoid redundant representations, we restrict attention to strategies satisfying $x_1 \geq x_2 \geq x_3$. Under this restriction, the strategy space is the triangle defined by the lines $x_1 = x_2$, $x_2 = 1-x_1-x_2$, and $1-x_1-x_2 = 0$. The vertices of this triangle are $(1,0)$, $(\tfrac{1}{2},\tfrac{1}{2})$, and $(\tfrac{1}{3},\tfrac{1}{3})$, corresponding to the strategies $(1,0,0)$, $(\tfrac{1}{2},\tfrac{1}{2},0)$, and $(\tfrac{1}{3},\tfrac{1}{3},\tfrac{1}{3})$, respectively.

In the following sections, we denote a general pure strategy by $(x,y,z)$ instead of $(x_1,x_2,x_3)$. The graphical representation of a pure strategy within the triangle is shown in Figure~\ref{fig:plot_triangle}.

\begin{figure}[ht]
    \centering
    \includegraphics[width=1\textwidth]{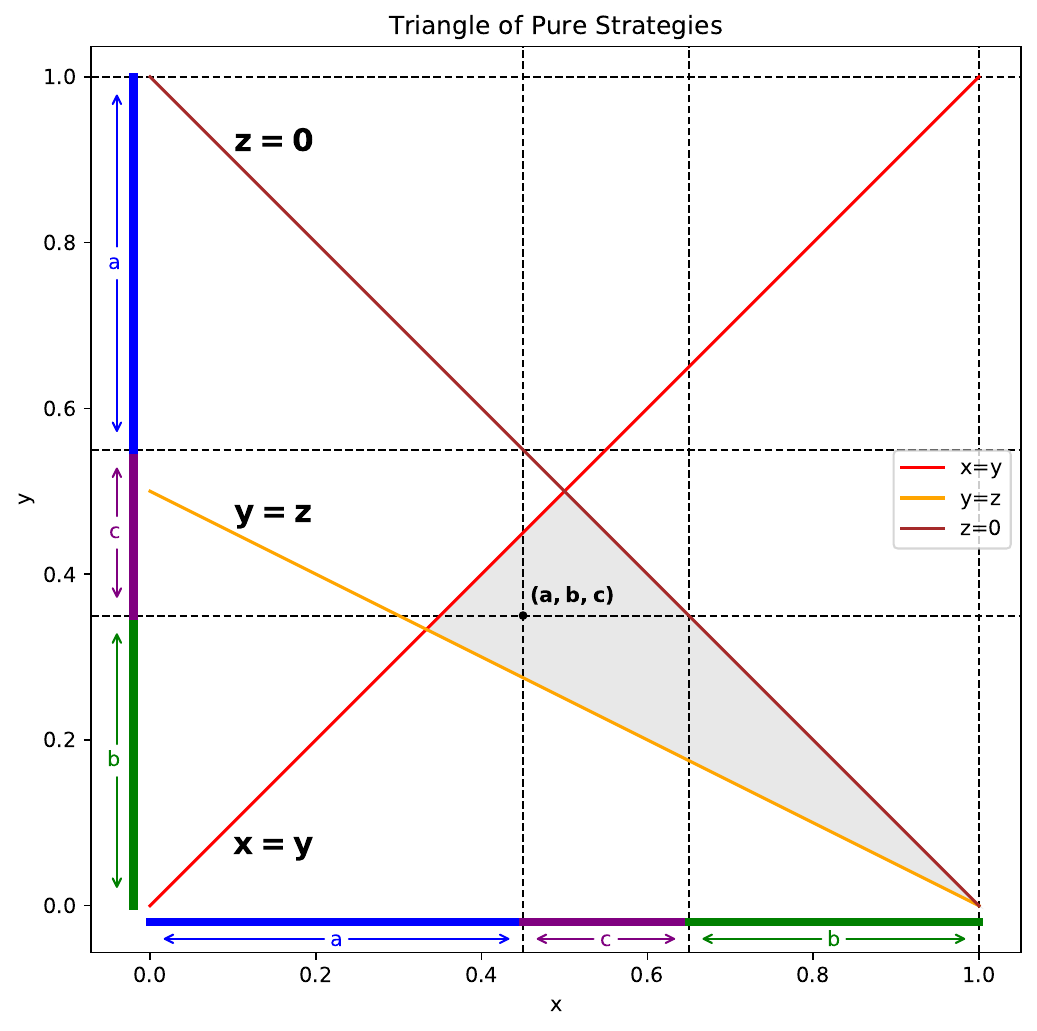}
    \caption{The graphical representation of a pure strategy $(a, b, c)$.}
    \label{fig:plot_triangle}
\end{figure}

\subsection{Numerical discretization}
\label{sec:discretization}

The infinite set of pure strategies in this game is defined as $S = \{(x, y, z) \in \mathbb{R}^3 : x \geq y \geq z \geq 0,\ x + y + z = 1\}$. To analyze the game numerically, we discretize this strategy space.

We define a two-dimensional grid with step size $\delta$. The finite set of pure strategies in the grid is 
$$S_\delta = \{(n_1\delta, n_2\delta)\in\mathbb{R}^2: n_1\delta \geq n_2\delta \geq 1-n_1\delta-n_2\delta \geq 0, n_1, n_2 \in \mathbb{Z}\}.$$

The number of pure strategies in $S_\delta$ depends on $\delta$ and increases on the order of $\tfrac{1}{\delta^2}$. In the following sections, we analyze the grid with $\delta = 0.01$, which contains 884 pure strategies, and is shown in Figure~\ref{fig:plot_grid}. 

\begin{figure}[ht]
    \centering
    \includegraphics[width=1\textwidth]{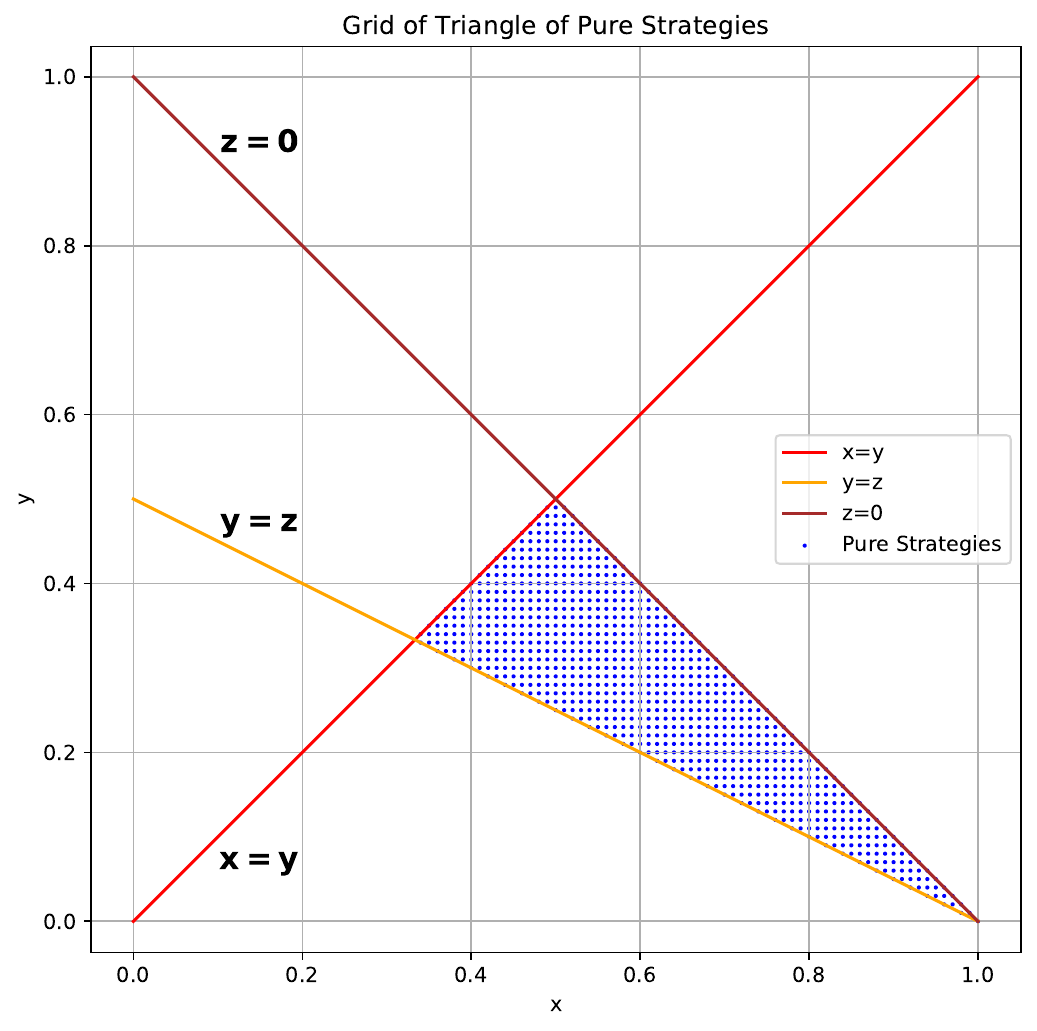}
    \caption{The grid of pure strategies with $\delta = 0.01$.}
    \label{fig:plot_grid}
\end{figure}

\subsection{Pure strategies}
\label{sec:pure}

As in the variation that maximizes expected total points, in the case of maximizing the winning probability, no pure strategy constitutes a symmetric equilibrium. In particular, we claim that the specific example in Lemma~\ref{thm:lemma1} still applies: let $(x,y,z)$ denote the pure strategy adopted by each of the other three players. Then, the strategy $(x - 2\epsilon, y + \epsilon, z + \epsilon)$ yields a strictly higher expected payoff for me than the strategy $(x,y,z)$, which yields a payoff of $0.25$.

This claim can be proved analytically. However, since a player's payoff in this variation depends on the points earned by the other players, we would need to consider up to $6^4=1296$ possible cases ($4$ players, each with $3!=6$ equally likely permutations of resource allocation). Although symmetry can reduce the number of distinct cases, it may be easier - and equally exact - to perform a complete enumeration using a simple computer program.

The results for all common strategies $(x, y, z)$ are shown in Table~~\ref{tab:strategies}. In each case, my payoff under a slightly modified strategy exceeds that of the common strategy, demonstrating that no such strategy constitutes a symmetric equilibrium.

\begin{table}[ht]
\centering
\[
\begin{array}{|c|c|c|}
\hline
\textbf{Their Strategy} & \textbf{My Strategy} & \textbf{My Payoff} \\ \hline
(x, y, z), x>y>z & (x-2\epsilon, y+\epsilon, z+\epsilon) & 570/1296 = 0.440 \\ \hline
(x, x, y), x>y & (x+\epsilon, x-2\epsilon, y+\epsilon) & 576/1296 = 0.444 \\ \hline
(x, y, y), x>y & (x-2\epsilon, y+\epsilon, y+\epsilon) & 888/1296 = 0.685 \\ \hline
(1/3, 1/3, 1/3) & (1/2, 1/2, 0) & 648/1296 = 0.500 \\ \hline
\end{array}
\]
\caption{Strategies in the four-player game.}
\label{tab:strategies}
\end{table}

Note that this result is independent of the numerical discretization and therefore remains valid in the original continuous setting.

\subsection{Mixed strategies}
\label{sec:mixed}
We have shown that there is no pure symmetric equilibrium. Moreover, for any pure strategy $(x,y,z)$ played by the other three players, I have a response that yields me a payoff of at least $0.440$. 

In this section, we analyze mixed symmetric strategies. Given a mixed strategy for the other three players, we compute my best response in the grid with $\delta = 0.01$. As a reminder, there is always a pure best response, and the grid contains a finite number of pure strategies. In addition, a strategy constitutes an equilibrium if and only if the best response against it yields a payoff of exactly $0.25$.

A mixed strategy on the grid assigns probabilities $0 \leq p_s \leq 1$ to each pure strategy $s \in S_\delta$, such that \(\sum_{s \in S_\delta} p_s = 1\). As noted in Section~\ref{sec:discretization}, for $\delta = 0.01$ we have $|S_\delta| = 884$. Therefore, computing the best response requires evaluating $884$ candidate pure strategies for me, each facing up to $884^3$ possible combinations of pure strategies for the other three players.

Due to computational constraints, an exact calculation is infeasible. Instead, we employ a Monte Carlo approach in which, for each candidate response, we simulate 10,000 realizations of the strategy combinations of the other players, sampled according to their mixed strategy, to approximate its expected payoff. This yields estimates that are sufficiently stable for our purposes.

Using this estimation procedure, we search for a mixed symmetric equilibrium. 

We test some natural candidates, motivated by the two-coordinate uniform strategy $(X,1-X)$ with $X \sim U[0,1]$. However, in analogy with the Bertrand paradox (Bertrand, 1889), where the concept of a "random chord" in a circle is not uniquely defined, the extension of the notion $X \sim U[0,1]$ to three coordinates lacks a unique interpretation.

\begin{example}
Consider the strategy that samples uniformly from the triangle with vertices $(1,0,0)$, $(0,1,0)$, and $(0,0,1)$. A similar strategy samples uniformly from the triangle with vertices $(1,0,0)$, $(\tfrac{1}{2},\tfrac{1}{2},0)$, and $(\tfrac{1}{3},\tfrac{1}{3},\tfrac{1}{3})$; in our context, the two differ only on a set of measure zero. Its best-response payoff on the grid is $0.362$.
\end{example}

\begin{example}
Consider a Blotto strategy that assigns equal probability to the two line segments from $(0,\tfrac{2}{3},\tfrac{1}{3})$ to $(\tfrac{2}{3},\tfrac{1}{3},0)$ and from $(0,\tfrac{1}{3},\tfrac{2}{3})$ to $(\tfrac{2}{3},0,\tfrac{1}{3})$. Its best-response payoff on the grid is $0.302$.
\end{example}

\begin{example}
Consider the strategy that samples $x,y,z \sim U[0,1]$ independently and plays $\bigl(\tfrac{x}{x+y+z}, \tfrac{y}{x+y+z}, \tfrac{z}{x+y+z}\bigr)$. Its best-response payoff on the grid is $0.299$.
\end{example}

Although these strategies are closer to equilibrium than any pure strategy, they are not sufficiently close. Consequently, a dedicated optimization algorithm is required to identify an equilibrium strategy.

\subsection{Algorithmic solution}
\label{sec:algorithmic}

We propose a variant of \emph{best-response dynamics} (e.g., Roughgarden, 2016), which we call \emph{approximate best-response dynamics with stochastic inertia}. This method combines elements from three established approaches: \emph{best-response dynamics with inertia} (e.g., Swenson et al., 2018), \emph{approximate best-response dynamics} (e.g., Bistritz and Leshem, 2017), and \emph{stochastic approximation} (e.g., Lai, 2003).

Best-response dynamics is an iterative procedure in which, at each step, a player updates their strategy to the best response against the strategies played in the previous round. Formally, if $x^{(t)}$ denotes the strategy profile at iteration $t$, then the update rule is
\[
x^{(t+1)} = \mathrm{BR}\bigl(x^{(t)}\bigr).
\]

In best-response dynamics with inertia, updates are smoothed through a convex combination of the previous strategy and its best response:
\[
x^{(t+1)} = (1-\alpha)\,x^{(t)} + \alpha \,\mathrm{BR}\bigl(x^{(t)}\bigr), \qquad \alpha \in (0,1).
\]

In approximate best-response dynamics, updates are smoothed through a uniform combination of all pure strategies within $\gamma$ of the best-response payoff:
\[
x^{(t+1)} = \widehat{\mathrm{BR}}(x^{(t)}), \qquad \gamma \ge 0.
\]

Finally, stochastic approximation algorithms form a general class of iterative methods, commonly used in optimization. Their convergence typically requires a decaying step-size sequence $(\alpha_t)$, satisfying:

\[
\sum_{t=0}^{\infty} \alpha_t = \infty, 
\qquad 
\sum_{t=0}^{\infty} \alpha_t^2 < \infty.
\]

We adopt this idea by replacing the fixed step size $\alpha$ with a decaying sequence $(\alpha_t)$, thereby incorporating stochastic inertia into the dynamics.

Our proposed dynamics synthesizes the strengths of these techniques in order to improve performance and encourage convergence in the game.

\begin{definition}
Let $G$ be the set of pure strategies in a symmetric $n$-player game, and let $\Delta(G)$ denote the set of mixed strategies over $G$. For $g_1,g_2 \in \Delta(G)$, define $u(g_1,g_2)$ as the payoff to a player who plays $g_1$ when each of the other $n-1$ players plays $g_2$.

\emph{Approximate Best-Response Dynamics with Stochastic Inertia} is an iterative procedure in which the mixed strategy $x^{(t)} \in \Delta(G)$ is updated at each iteration according to the rule:
\[
x^{(t+1)} = (1 - \alpha_t) \cdot x^{(t)} + \alpha_t \cdot \widehat{\mathrm{BR}}(x^{(t)}),
\]
where:
\begin{itemize}
    \item $\alpha_t \in (0,1)$ is a decaying step-size parameter,
    \item $\widehat{\mathrm{BR}}(x^{(t)})$ is an approximate best-response distribution that assigns equal probability to all pure strategies $g \in G$ satisfying
    \[
    u(g,\ x^{(t)}, \dots,\ x^{(t)}) \ge \max_{g' \in G} u(g',\ x^{(t)}, \dots,\ x^{(t)}) - \gamma,
    \]
    for some tolerance parameter $\gamma \ge 0$.
\end{itemize}
\end{definition}

A strategy $x$ is an \emph{$\varepsilon$-equilibrium} if the maximum gain that a player can obtain by unilaterally deviating from it is bounded by $\varepsilon$. Formally:
\[
\max_{g \in G} u(g,\ x) - u(x, x) \leq \varepsilon.
\]  

Returning to our game, we look for an $0.01$-equilibrium, which corresponds to a best-response payoff of less than $0.26$, since the equilibrium payoff is $0.25$. We initialize the process with a decaying step-size $\alpha_t = \frac{1}{t}$, a tolerance $\gamma = 0.01$, and an initial strategy $x^{(0)}$ that assigns equal probabilities to all pure strategies $s \in S_\delta$.

We ran the \emph{Approximate Best-Response Dynamics with Stochastic Inertia} algorithm for 1000 iterations. A strategy achieving a best-response payoff below $0.26$ is first reached after approximately 200 iterations. The best strategy obtained during the process achieves a best-response payoff of around $0.257$. The convergence of the best-response value over iterations is shown in Figure~\ref{fig:plot_iterations}.

\begin{figure}[ht]
    \centering
    \includegraphics[width=1\textwidth]{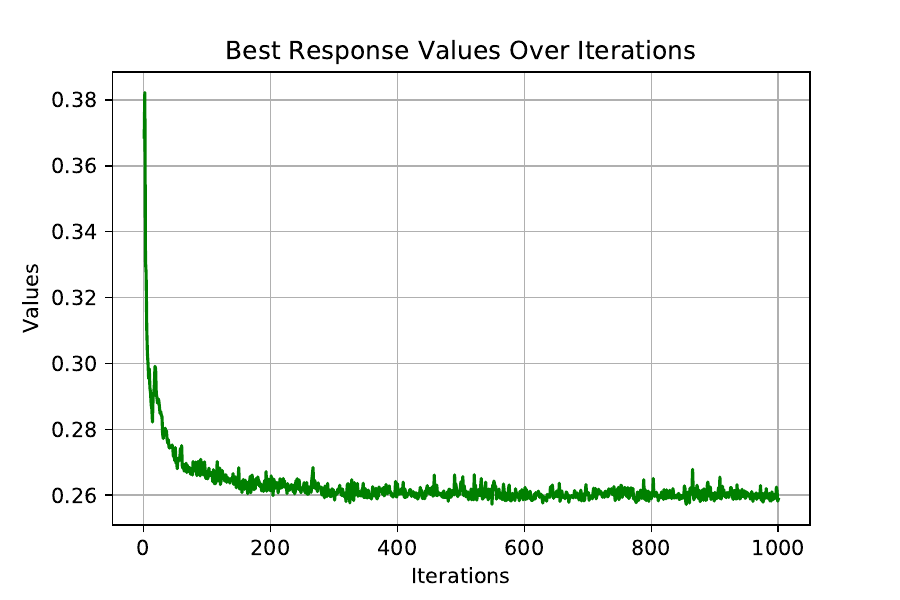}
    \caption{Convergence behavior of the algorithm.}
    \label{fig:plot_iterations}
\end{figure}

Our Monte Carlo approach with 10,000 simulations is sufficiently stable for the update step of the algorithm. However, in order to better evaluate the best returned strategy, we increased the number of simulations to 1,000,000, and find that its best-response payoff is around $0.254$. 

The returned strategy is characterized by the probabilities assigned to each of the 884 pure strategies in $S_\delta$. These probabilities are shown graphically in Figure~\ref{fig:optimal_strategy}, normalized so that the highest probability assigned is scaled to 1. A discretized version, in which the probability of each strategy is rounded to three decimal places, thus creating six levels of positive probabilities, is provided in the appendix. 

\begin{figure}[ht]
    \centering
    \includegraphics[width=1\textwidth]{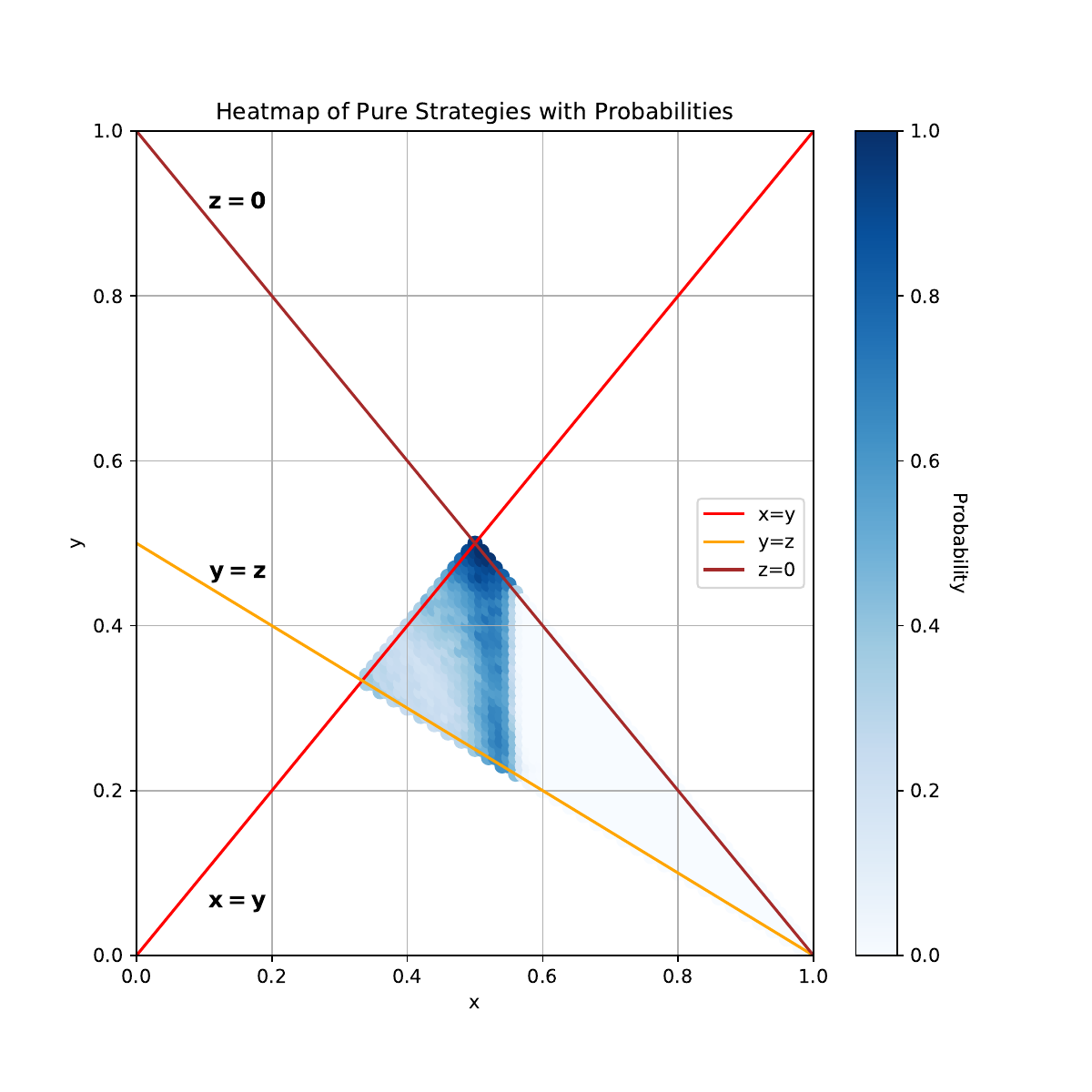}
    \caption{Normalized probability weights assigned by the returned strategy.}
    \label{fig:optimal_strategy}
\end{figure}

An additional interesting graphical characteristic of a strategy is given by its marginal distribution: When Player A faces Player B (who plays the strategy), what is the probability density over the resource that Player B allocates to the match against Player A? The estimated marginal density for the returned strategy is shown in Figure~\ref{fig:optimal_strategy_dist}.

\begin{figure}[ht]
    \centering
    \includegraphics[width=1\textwidth]{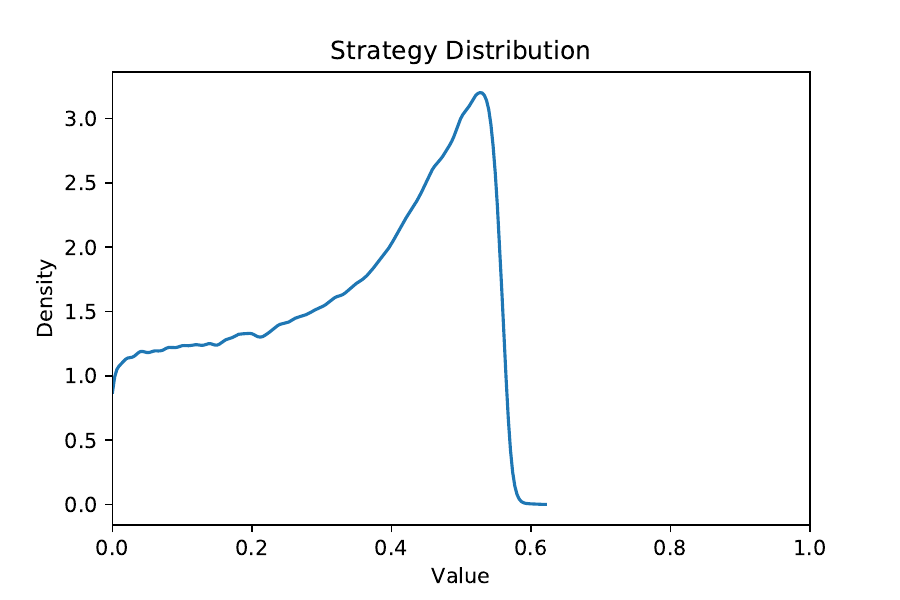}
    \caption{Marginal distribution of returned strategy $n=4$.}
    \label{fig:optimal_strategy_dist}
\end{figure}

The marginal density increases up to a very sharp drop to zero. This increase is not linear; it is well described either by a quadratic function $ax^2 + bx + c$ with explained variance $R^2 = 0.98$, or by an exponential function $\exp(ax + b) + c$ with explained variance $R^2 = 0.99$. 

In addition, the strategy with a sharp drop to zero is better than a similar strategy with an instantaneous drop to zero, in terms of its best-response payoff being closer to the equilibrium payoff.

\subsection{Five players}
\label{sec:five}

The algorithmic solution presented in the previous section extends naturally from the four-player case to any number of players, though with two main limitations. 

First, while a geometric representation of strategies is feasible for four players in two dimensions, it requires a three-dimensional representation for five players, and becomes practically infeasible for six players or more. 

Second, the computational complexity increases rapidly with the number of players. For instance, with a discretization step-size of $\delta = 0.01$, the four-player game involved 884 pure strategies, whereas the five-player game requires 8,037 pure strategies. 

Nevertheless, implementing the algorithm remains possible, and the marginal density of the computed strategy can still be presented. We ran the algorithm in the five-player game for approximately 1250 iterations until reaching a strategy with a payoff below $0.21$, noting that the equilibrium payoff is $0.20$, and a strategy with a payoff below $0.215$ was first reached after approximately 250 iterations. 

The marginal density of the resulting strategy increases non-linearly before exhibiting a sharp drop to zero, mirroring the behavior observed in the four-player case. It is shown in Figure~\ref{fig:optimal_strategy_dist_5p}.

\begin{figure}[ht]
    \centering
    \includegraphics[width=1\textwidth]{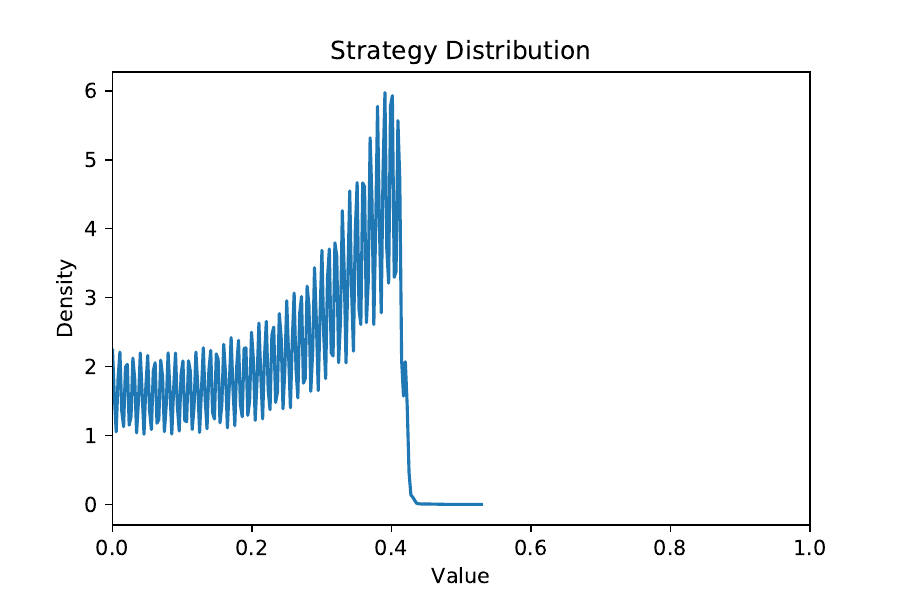}
    \caption{Marginal distribution of returned strategy $n=5$.}
    \label{fig:optimal_strategy_dist_5p}
\end{figure}

\section{Conclusion}

We defined a round-robin resource allocation contest and analyzed four variations. The resulting symmetric equilibrium strategies are summarized in Table~\ref{tab:results}, with pure equilibrium strategies in the variations with three players and mixed equilibrium strategies in the variations with four or more players. In addition, we proposed a best-response dynamics algorithm and used it to solve one of the variations.

\begin{table}[ht]
\centering
\begin{tabular}{|c|c|c|}
\hline
 & $n=3$ players & $n \geq 4$ players \\ \hline
\makecell{Total points \\ payoff} & \makecell{$(x, 1-x)$ \\ $\forall x \in \left[\frac{1}{2}, 1\right]$} & \makecell{Any Blotto strategy with \\ marginal dist. $\mathrm{U}[0, \tfrac{2}{n-1}]$} \\ \hline
\makecell{Winner-takes-all \\ payoff} & \makecell{$(\tfrac{1}{2}, \tfrac{1}{2})$ \\ unique solution} & \makecell{Solved numerically for \\ $n=4$ and $n=5$} \\ \hline
\end{tabular}
\caption{Symmetric equilibrium results.}
\label{tab:results}
\end{table}

The table illustrates a key difference between the two payoff schemes. One can think of the total points payoff as representing a tournament where players receive a prize for every match win, while in the winner-takes-all payoff, players receive a prize only for winning the entire tournament.

The tension between maximizing the total score and maximizing the probability of outscoring the opponents also exists in the scoring-based games themselves.  Skinner (2011) modeled a game as a sequence of independent plays and found that, in certain situations, it may be optimal for teams to select plays with a lower expected score but a higher variance in scoring outcomes. 

Although we were able to find symmetric equilibrium strategies in all variations, some questions remain open.

In the total points payoff with four or more players, we proved that any Blotto strategy is a symmetric equilibrium. The equilibrium is not unique, as there are many different Blotto strategies. However, it is unclear whether there is another symmetric equilibrium strategy, distinct from a Blotto strategy.

In the winner-takes-all payoff with three players, we proved that $(\tfrac{1}{2}, \tfrac{1}{2})$ is the unique symmetric equilibrium. Interestingly, while it is a best response against any common strategy of the other two players, it may not be a best response in the general case. For example, in the settings of Lemma~\ref{thm:main0}, consider the case $x>y>0.5$. In that case, it is more beneficial for the third player to play $z=y$ and receive $P_z=\frac{11}{24}$, rather than play $z=0.5$ and receive $P_z=\frac{5}{12}$. Further analysis may result in interesting insights.

In the winner-takes-all payoff with four or more players, we proposed an algorithm and implemented it in the game with four players. We presented a simple version with fixed tolerance $\gamma = 0.01$, with which the process stabilized after $1000$ iterations and returned a strategy that achieved a best-response payoff of $0.254$. However, an improved version with a decreasing tolerance $\gamma_t = \frac{0.05}{\sqrt{t}}$ was able to return a strategy with a best-response payoff of $0.252$, cutting $50\%$ of the gap toward the equilibrium payoff, after $5000$ iterations. Further optimization of the algorithm parameters may lead to even better results.

In addition to these questions, there are more potential variations to explore. A fundamental assumption in our model is that players allocate resources in advance, before the competition begins. Analyzing a game with dynamic strategies that allocate resources to stage $k$ only after knowing the results of stage $k-1$ would enrich the model. 

Furthermore, one could define an alternative payoff scheme in which there is a reward for winning the tournament (finishing first) and a penalty for losing it (finishing last). Another variation may consider using the Tullock mechanism to determine match outcomes, instead of the all-pay rule.

Finally, one may also extend the analysis from a symmetric game to an asymmetric game, where players have heterogeneous resource budgets. Apart from our two payoff schemes, a natural framework in this case is to assign players a common objective function that combines two components: a large prize for winning the tournament and smaller prizes for each match victory. In such a setting, it is plausible that stronger players would focus on maximizing their probability of winning the tournament, while weaker players would instead emphasize maximizing their expected total points.

\appendix

\section{Proof of Lemma~\ref{thm:main0}}

There are eight configurations, denoted A, B, C, D, E, F, G, and H. In configurations A, B, C, and D, Player 1 allocates $x$ against Player 2. In configurations A, C, E, and G, Player 2 allocates $y$ against Player 1. In configurations A, B, E, and F, Player 3 allocates $z$ against Player 1. The complete payoff table, containing all 32 case–configuration combinations (except the trivial case $x=y=z$), is provided below.

\begin{longtable}{|c|c|c|c|c|}
\hline
\textbf{case} & \textbf{configuration} & \textbf{$P_x$} & \textbf{$P_y$} & \textbf{$P_z$} \\ \hline
\endfirsthead 

\multicolumn{5}{c}{\textbf{Table of case-configuration combinations (Continued)}} \\
\hline
\textbf{case} & \textbf{configuration} & \textbf{$P_x$} & \textbf{$P_y$} & \textbf{$P_z$} \\ \hline
\endhead

$x > y > z$ & A & 0 & 0 & 1 \\ \hline
$x > y > z$ & B & 1/3 & 1/3 & 1/3 \\ \hline
$x > y > z$ & C & 0 & 0 & 1 \\ \hline
$x > y > z$ & D & 1/3 & 1/3 & 1/3 \\ \hline
$x > y > z$ & E & 1/3 & 1/3 & 1/3 \\ \hline
$x > y > z$ & F & 0 & 1 & 0 \\ \hline
$x > y > z$ & G & 1/3 & 1/3 & 1/3 \\ \hline
$x > y > z$ & H & 0 & 1 & 0 \\ \hline

$x = y > z$ & A & 0 & 0 & 1 \\ \hline
$x = y > z$ & B & 1/3 & 1/3 & 1/3 \\ \hline
$x = y > z$ & C & 0 & 0 & 1 \\ \hline
$x = y > z$ & D & 1/3 & 1/3 & 1/3 \\ \hline
$x = y > z$ & E & 1/3 & 1/3 & 1/3 \\ \hline
$x = y > z$ & F & 1/2 & 1/2 & 0 \\ \hline
$x = y > z$ & G & 1/3 & 1/3 & 1/3 \\ \hline
$x = y > z$ & H & 1/2 & 1/2 & 0 \\ \hline

$x > y = z > 1/2$ & A & 0 & 0 & 1 \\ \hline
$x > y = z > 1/2$ & B & 1/3 & 1/3 & 1/3 \\ \hline
$x > y = z > 1/2$ & C & 0 & 0 & 1 \\ \hline
$x > y = z > 1/2$ & D & 0 & 0 & 1 \\ \hline
$x > y = z > 1/2$ & E & 0 & 1 & 0 \\ \hline
$x > y = z > 1/2$ & F & 0 & 1 & 0 \\ \hline
$x > y = z > 1/2$ & G & 1/3 & 1/3 & 1/3 \\ \hline
$x > y = z > 1/2$ & H & 0 & 1 & 0 \\ \hline

$x > y = z = 1/2$ & A & 0 & 0 & 1 \\ \hline
$x > y = z = 1/2$ & B & 0 & 0 & 1 \\ \hline
$x > y = z = 1/2$ & C & 0 & 0 & 1 \\ \hline
$x > y = z = 1/2$ & D & 0 & 0 & 1 \\ \hline
$x > y = z = 1/2$ & E & 0 & 1 & 0 \\ \hline
$x > y = z = 1/2$ & F & 0 & 1 & 0 \\ \hline
$x > y = z = 1/2$ & G & 0 & 1 & 0 \\ \hline
$x > y = z = 1/2$ & H & 0 & 1 & 0 \\ \hline

\end{longtable}

\section{Proof of Remark~\ref{thm:main2}}

Assume that the other two players, denoted Opponent 1 and Opponent 2, allocate their resources independently as $(X, 1-X)$, where $X \sim f_X$, a continuous density function over $[0,1]$. Assume that I choose my allocation deterministically as $(y, 1-y)$, where $y \in [\frac{1}{2},1]$. 

We will prove that my payoff, denoted as $P_0$, is given by 
\[
P_0 = \frac{1}{6} + a_y (1 - a_y) \quad \text{where} \quad a_y = \frac{1-\Pr(1-y < X < y)}{2}.
\]

Without loss of generality, assume that I allocate $y$ against Opponent 1, and $1-y$ against Opponent 2.

The other two players allocate their resources independently as $(X, 1-X)$, where $X \sim f_X$. Therefore, against each of them, I meet in probability $\frac{1}{2}$ an allocation $X \sim f_X$, and in probability $\frac{1}{2}$ an allocation $X \sim f_{1-X}$.

In other words, the allocation I face is $X \sim g_X$, where $g_X = \frac{f_X+f_{1-X}}{2}$, and $g_X$ is symmetric around $\frac{1}{2}$, as $g_X = g_{1-X}$. Therefore, $\Pr(X > y \mid X \sim g_X) = \Pr(X < 1-y \mid X \sim g_X)$. We will mark $\Pr(X > y \mid X \sim g_X) = \Pr(X < 1-y \mid X \sim g_X)=a_y$, and $\Pr(1-y < X < y \mid X \sim g_X) = 1-2a_y$. 

We will mark the two independent realizations of the allocations of the other two players against me as $x_1$ and $x_2$. Meaning, $x_1, x_2 \sim g_X$. 

With probability $1$, the values $x_1$, $1-x_1$, $x_2$, $1-x_2$, $y$, and $1-y$ are all distinct, so there are no draws in the matches. Either one player wins two matches and, therefore, wins the tournament, or there is a triple draw in the tournament, meaning that each player wins one match.

I win the tournament if I win both matches. That is, if $y>x_1$ and $1-y>x_2$. This happens in probability $a_y(1-a_y)$.

There is a triple draw if $(A)$ I win against Opponent 2 who wins against Opponent 1 who wins against me, or $(B)$ I win against Opponent 1 who wins against Opponent 2 who wins against me.

$(A)$ occurs with probability $\Pr(1-y > x_2 \text{ and } 1-x_2 > 1-x_1 \text{ and } x_1 > y) = \Pr(x_1 > y \text{ and } x_2 <1-y) = \Pr(x_1 > y)\cdot\Pr(x_2 <1-y) = a_y \cdot a_y = a_y^2$

$(B)$ occurs with probability $\Pr(y > x_1 \text{ and } 1-x_1 > 1-x_2 \text{ and } x_2 > 1-y) = \Pr(0 < x_1 < y \text{ and } 1-y < x_2 < 1 \text{ and } x_1 < x_2)$. The event can be decomposed into a disjoint union of three sub-events:

\begin{itemize}
    \item The first sub-event is $0 < x_1 < 1-y$ and $1-y < x_2 < 1$.
    \item The second sub-event is $1-y < x_1 < y$ and $1-y < x_2 < y$ and $x_1 < x_2$.
    \item The third sub-event is $1-y < x_1 < y$ and $y < x_2 < 1$.
\end{itemize}

The first sub-event occurs in probability $a_y(1-a_y)$, the second sub-event occurs in probability $\frac{1}{2}\cdot(1-2a_y)^2$, and the third sub-event occurs in probability $(1-2a_y)a_y$. The sum of the probability of all options is $\frac{1}{2}-a_y^2$.

In conclusion, the probability of a triple draw (which results in a payoff of $\frac{1}{3}$ for me) is $\frac{1}{2}$, and my payoff is $P_0 = \frac{1}{6} + a_y (1 - a_y)$ for $a_y = \Pr(X > y \mid X \sim g_X)$. Since $g_X = \frac{f_X+f_{1-X}}{2}$, we have $a_y = \frac{1}{2}\Pr(X > y \mid X \sim f_X)+\frac{1}{2}\Pr(1-X > y \mid X \sim f_X)=\frac{1}{2}\Pr(X < 1-y  \text{ or } X > y \mid X \sim f_X)=\frac{1}{2}(1-\Pr(1-y<X<y \mid X \sim f_X))$. 

\hfill \(\square\)

\section{Discretized version of the returned strategy}

\begin{figure}[H] 
    \centering
    \includegraphics[width=1\textwidth]{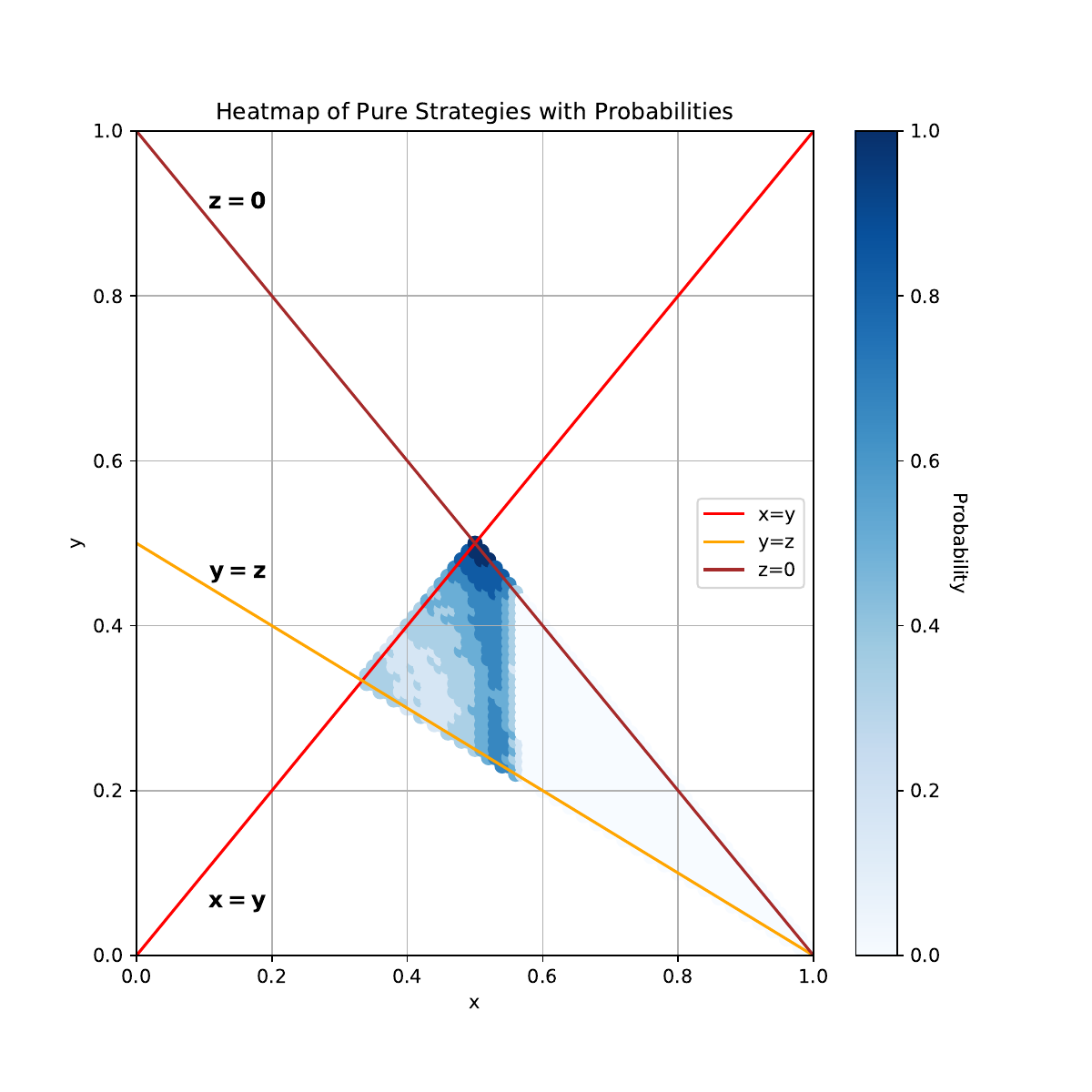}
    \caption{Normalized and discretized probability weights assigned by the returned strategy.}
    \label{fig:jumped_optimal_strategy}
\end{figure}

\end{document}